\documentclass[12pt]{amsart}
\usepackage{amsmath,amsfonts,amsthm,amssymb,amscd,url}
\usepackage{dsfont}
\usepackage{enumerate}
\usepackage[margin=1in]{geometry}
\usepackage{multirow,booktabs}

\usepackage{bigints}

\DeclareFontFamily{OT1}{rsfs}{}
\DeclareFontShape{OT1}{rsfs}{n}{it}{<-> rsfs10}{}
\DeclareMathAlphabet{\mathscr}{OT1}{rsfs}{n}{it}

\newtheorem{prop}{Proposition}[section]
\newtheorem{theorem}[prop]{Theorem}

\newtheorem{corollary}[prop]{Corollary}
\theoremstyle{theorem}
\newtheorem*{AH}{The Alternative Hypothesis (AH)}
\newtheorem{lemma}[prop]{Lemma}
\newtheorem{Question}[prop]{Question}

\newtheorem*{defn*}{Definition}
\theoremstyle{definition}
\newtheorem{Rem}{Remark}
\numberwithin{equation}{section}
\author{Farzad Aryan}
\title{A new approach to gaps between zeta zeros}

\begin{document}
\maketitle
\begin{abstract}
We study the value-distribution of Dirichlet polynomials on the critical line $\Re(s)=\tfrac{1}{2}$. As a consequence, we prove a corollary on small consecutive gaps between zeros of the Riemann zeta function. We also examine the distribution of zeros under the so-called alternative hypothesis and present a new approach to the problem of gaps between the zeros. 
\end{abstract}
\section{\bf Introduction.}
Studying the zeros of the Riemann zeta function is one of the central themes in analytic number theory. The Riemann hypothesis, which remains unsolved for more than a century, predicts that  all of the non-trivial zeros of the zeta function lie on the critical line in the complex plane. Having all the zeros on the critical line leads us to questions about distribution of gaps between zeta zeros. \\

In this paper we study the value-distribution of Dirichlet polynomials with respect to the size of gaps between the zeros of the Riemann zeta function. One of our objectives is to introduce a new approach that we believe has a more reasonable chance to resolve the problem of ``alternative
hypothesis" (AH) on the distribution of zeros of the zeta function. AH, formulated by Conrey in~\cite{ConAH}, 
models the distribution that the Riemann zeta zeros would
have if Landau-Siegel zeros were to exist. According to Farmer, Gonek and Lee~\cite{FGL}
``AH is obviously absurd, but it has not been disproved. A sufficiently strong disproof would
show that there are no Landau-Siegel zeros". \\

To explain these terms and motivate our
investigation, we begin with the class number of imaginary quadratic
fields. The class number, in a way, would measure the failure of unique factorization in the ring of integers of a number field. For example in $\mathbb{Z}[\sqrt{-5}]$, which is the ring of integers of $\mathbb{Q}[\sqrt{-5}],$ we have that $6$ can be written as a product of irreducible in two different ways: $2\times3$ and $(1+\sqrt{-5})\times(1-\sqrt{-5}).$ The class number of this field is two. \\

An important  question is how big the class number can get. The answer is very much dependant on the possible existence of so called Landau-Siegel zeros for Dirichlet $L$-functions. These zeros are possible counterexamples to the generalized Riemann hypothesis which are on the real line. More precisely a Dirichlet $L$-function attached to a real character $\chi$ of conductor $q$ might have a real zero within a distance $\log^{-1}q$ from $s=1$. \\

The possibility of Landau-Siegel zeros has an unfortunate effect in the class number formula. They can force $L(1,\chi),$ and hence the class number, to be very small. It took a significant effort by Goldfeld, Gross and Zagier~\cite{Gol, Gr-Za} to show that the class number gets arbitrary large. This was achieved by introducing the $L$-functions attached to elliptic curves into the problem.  In light of this, there is a significant incentive to eliminate this possibility and there has been much effort in number theory in this direction.  \\

To explain the connection with zeros of the Riemann zeta function, it was noticed by Montgomery~\cite{Mo} that small class numbers would imply that zeros of the Riemann zeta function are rigidly spaced. Conrey and Iwaniec  in~\cite{CI} provided a detailed analysis of this. They show that existence of many pairs of consecutive zeros of the zeta function with gap smaller than $0.5$ times the average gap would imply that there are no Landau-Siegel zeros. \\

Over the years this problem attracted much attention and resulted in considerable progress in the theory of  $L$-function, including Montgomery's pair correlation conjecture. If we normalize zeros in a way that their average gap equals $1,$ Montgomery's work, under the assumption of the Riemann hypothesis, would imply that there are infinitely many gaps smaller than $0.68.$ His conjecture, though predicts that the size of the gaps between consecutive zeros of the zeta function can be arbitrary small and arbitrary large. With another method Montgomery and Odlyzko~\cite{MO} showed that there are normalized gaps smaller than $0.5179.$  The caveat with this method is that it is unable to prove existence of normalized gaps smaller than $0.5.$ \\

Montgomery and Odlyzko's result was improved to $0.5172$ by Conrey, Ghosh and Gonek~\cite{CGG} and subsequently reduced to to $0.5155$ by Bui, Milionovich and Ng~\cite{BMN}, all assuming the Riemann hypothesis. Further improvements was achieved in \cite{Pre, FWu} by Feng and Wu and Preobrazhenski\u i. 
\\


Returning to the problem of the effect of Landau-Siegel zeros on the distribution of zeros of the Riemann zeta function, from the work of Conrey and Iwaniec~\cite{CI} one can deduce that the existence of these zeros would imply the normalized gap between zeros of the zeta function are close to being half integers. This was further explored by Conrey in \cite{ConAH} and Heath-Brown in \cite{HeaAH}  and they proposed an alternative  hypothesis (to the pair correlation conjecture). Here we state the formulation from the paper of Farmer, Gonek and Lee~\cite{FGL}: Let $\gamma$ be an imaginary part of the a zero of the zeta function. The number of zeros up to the height $T$ is about $\displaystyle{(2\pi)^{-1}T \log T},$ which makes the average gap about $2\pi\log^{-1} T.$ We normalize them by setting $$\tilde{\gamma}= \frac{1}{2\pi}\gamma \log\big(\frac{\gamma}{2\pi}\big),$$
and by $\gamma^{+}$ and $\gamma^{-}$ we mean the zero after and before $\gamma$ respectively.   \\

\begin{AH}  There exists a real number $T_0$ such that if $\gamma > T_0$ , then
$$\tilde{\gamma}^{+}- \tilde{\gamma} \in \tfrac{1}{2}\mathbb{Z}.$$
That is, almost all the normalized neighbor spacings are an integer or half-integer.\\
\end{AH}
Let $g_{\sigma}$  denote the proportion of normalized gaps between consecutive zeros that are equal to $\sigma.$ Farmer, Gonek and Lee~\cite{FGL}, assuming the AH, showed that $g_{0.5}\approx 0.297$ and $4/\pi^2 \leq g_{1} \leq  0.5$. Our first theorem provides more information on the density of gaps under AH.  
\begin{theorem}
\label{Prop6} Assume the Riemann hypothesis and AH. We have that
\begin{itemize}
\item $g_{1.5} \geq 0.1079271,$
\item $g_{1.5} \geq 0.1+2(0.5-g_1+ 0.0039635), $
\item for $k \geq 2$ we have $g_{k}   \leq \frac{0.18951-  2(0.5-g_1)}{k}$,
\item if $g_1= \frac{4}{\pi^2}$ then $g_2, g_{2.5}, g_3 , \cdots \approx 0.$
\end{itemize}

\end{theorem}\begin{Rem}
The first and a slightly weaker version of the second part of the theorem also appeared in Baluyot's thesis~\cite{Bal-Thesis}, also see \cite{Sie}.
\end{Rem}
Recently Tao~\cite{Tao} and independently Lagarias and Rodgers~\cite{LR}  constructed a sequence of points in $\tfrac{1}{2}\mathbb{Z}$, that satisfies what is currently known about correlations between zeta zeros. In other words they show that AH cannot be ruled out with the knowledge that we currently posses about pair correlation of zeta zeros. Their works describe a conjectural model they call AGUE. Among other useful information, the model  precisely predict the density of gaps between zeros under AH. 
For example, here is a list of information we can get from AGUE:

\begin{itemize}
\item $g_{0.5} \simeq 0.297$
\item $g_{1} \simeq 0.453$
\item $g_{1.5} \simeq 0.207$
\item $g_{2} \simeq 0.0397$
\item $g_{2.5}\simeq 0.003$
\end{itemize}
It would also gives the density of two consecutive gap of certain sizes. For example the Lebesgue measure of two connective gaps of length $0.5$ is about $0.05.$ \\

The above prediction falls very tight with bounds we got in  Theorem~\ref{Prop6}. 
Basically our theorem says that if we set $g_1=0.453$ (as predicted by AGUE), we must have $g_{1.5} \geq 0.201,$ and $g_2 \leq 0.048.$ Another interesting conclusion of our theorem is that if $g_1\approx 0.4,$ then about $30\%$ of gaps are of length $0.5$ and the remaining $30\%$ are of length $1.5$. In other words, statistically we will not get any gap bigger than $1.5$ if $g_1\approx 0.4,$ . \\

Now we introduce our method to study value-distribution of Dirichlet polynomials. 

\section{\bf Value distribution of Dirichlet polynomials} \label{ValueDis} In this section we state our main theorem and an application to the problem of gaps between zeta zeros. Our theorem estimates the first moments of certain test functions with respect to various probability measures we get form Dirichlet polynomials.\footnote{This probability measures are often referred to as mollifiers in many places in the literature. In \cite{2nd-Pap} we look at the higher moments of these test functions with respect to these measures.} To explain our results easier let us  switch to probability language. For $a, b >0,$ consider the probability measure
\begin{equation}
\label{def-Mu}
\displaystyle{\mu_A((a,b]):= \frac{\int_{a}^{b}\omega(\tfrac{1}{2}+it)\big|A(\tfrac{1}{2}+it)\big|^2}{\int\omega(\tfrac{1}{2}+it)\big|A(\tfrac{1}{2}+it)\big|^2},}
\end{equation}
where 

\begin{equation}
\label{def_A}
A(s)=\sum_{n<T^{1-\epsilon}} a(n)n^{-s}
\end{equation} is a Dirichlet polynomial and $\omega$ a cut-off weight, centered around $T,$ with $\parallel \omega\parallel_1 =1$. For $A(s)=1,$ the above gives the Lebesgue measure. The measure $\mu_A$ sometimes is called a mollifier. \\

Let us give precise definitions of test functions we will use. 
\begin{defn*}
Let $\alpha \in \mathbb{R}^{+}$ and let $\gamma$ be an imaginary part of a zero of the zeta function. Define
\begin{equation}
\label{C_{1, alpha}}
C_{1, \alpha}(t):= -\alpha^{-1} +\sum_{\gamma} \bigg(\frac{\sin(\tfrac{\alpha}{2}(\gamma-t)\log T)}{\tfrac{\alpha}{2}(\gamma-t)\log T}\bigg)^2.
\end{equation}
\begin{equation}
\label{C_{2, alpha}}
C_{2, \alpha}(t):= -\alpha^{-1} + \sum_{\gamma} \bigg(\frac{\sin(\tfrac{\alpha}{2}(\gamma-t)\log T)}{\tfrac{\alpha}{2}(\gamma-t)\log T}\bigg)^2\Big(\tfrac{1}{1-\big(\tfrac{\alpha}{2\pi}(t-\gamma)\log T\big)^2}\Big).
\end{equation}
\end{defn*}
These test functions come form Fourier pairs:
\begin{align}
\label{defC}
  &  C_1(u)=1-u \hspace{24 mm} \text {and } \hat{C}_{1}(v)= \big(\frac{\sin(\pi v)}{\pi v}\big)^2 \\ & C_2(u)=1-u + \frac{\sin(2\pi u)}{2\pi} \hspace{3 mm} \text {and } \hat{C}_{2}(v)= \big(\frac{\sin(\pi v)}{\pi v}\big)^2\big(\frac{1}{1-v^2}\big)
\end{align}
These are functions with large values where we have an accumulation of zeros and it is negative around the large gaps between the zeros. One distinction between them is that, If we assume AH and also all zeros are simple, for $\alpha \geq 1,$  $C_{1, \alpha}$ is always smaller than one, but $C_{2, \alpha}$ can get slightly bigger than one.\footnote{We give the proof of these small facts in the last section in Lemma \ref{8.2}.} Our main theorem estimates the $\mu_A$-mollified mean of these test functions:
\begin{theorem}
\label{Main-Th}
Define $\lambda_{\beta_1, \beta_2}$ to be a completely multiplicative function with  $\lambda_{\beta_1, \beta_2}(p)=-1$ for $ T^{\beta_1}< p \leq T^{\beta_2}$ and $\lambda_{\beta_1, \beta_2}(p)=1$ otherwise. Furthermore, let $\mu_{A_{\beta_1, \beta_2, r , \eta}}$ be the measure we build, as in \eqref{def_A}, using

\begin{equation}
\label{def_an}
a(n)= \lambda_{\beta_1, \beta_2}(n)d_{r}(n)\big(1-\frac{\log(n)}{\log T}\big)^{\eta},
\end{equation}  
where $d_r$ is the generalized divisor function. Assuming the Riemann hypothesis, for $i=1, 2$ we have 
\begin{align}
\notag \int C_{i, \alpha}\big(t & +\tfrac{2\pi d}{\log T}\big)  d\mu_{A_{\beta_1, \beta_2, r , \eta}} +O(\tfrac{1}{\log T})=  \\ & \notag \frac{2r \int_{0}^{\min(\alpha, 1)}\int_{0}^{1-u} \lambda_{\beta_1, \beta_2}(u) \cos(2\pi d u) v^{r^2-1}C_i(u/\alpha)(1-v)^\eta(1-u-v)^\eta du dv }{\alpha\int_{0}^{1}v^{r^2-1}(1-v)^{2\eta}dv}.
\end{align}
where $\lambda_{\beta_1, \beta_2}(u)=-1$ for $\beta_1 <u \leq \beta_2$ and equals $1$ otherwise. For $i=1, 2,$ $C_i$ are defined in \eqref{defC}.
\end{theorem}
Now we give a corollary to show a general application of Theorem \ref{Main-Th}.
\begin{corollary}
\label{Corr1}
Assume the Riemann hypothesis, then either we have infinitely many zeta zeros of height asymptotic to $ T$, that spaced smaller than the half of the average gap, i.e.  $$\gamma_{n+1}-\gamma_n < \frac{\pi}{\log T},$$ or we have 
\begin{equation}
    \gamma_{n+2}-\gamma_n < 1.181 \hspace{1 mm}\frac{2\pi}{\log T}.
\end{equation}
\end{corollary}

\begin{Rem} This result should be compared with the result of \cite{Con-Butt} that shows 
\begin{equation*}
    \gamma_{n+2}-\gamma_n < 1.576 \hspace{1 mm}\frac{2\pi}{\log T},
\end{equation*}
infinitely often. 
\end{Rem} 
We will provide the proof of the above results in the next section . In section \ref{refine} we will introduce our refinement of Montgomery-Odlyzko's method. 
  \section{\bf Proof of Theorems \ref{Main-Th} and Corollary \ref{Corr1}}
We start with the proof of Theorem \ref{Prop6} on the effect of AH on the density of gaps between zeros of the Riemann zeta function.
 \begin{proof}[ Proof of Theorem \ref{Prop6}]
 We will use the simple fact that under AH we have
\begin{equation}
\label{6.1}
\sum_{k\in \tfrac{1}{2}\mathbb{N}} kg_{k} =1.
\end{equation}
This holds since these gaps must cover the whole interval under consideration, in other words, $$\sum_{\sigma} \mu_{\text{Leb}}( \text{gaps of length   } \sigma)=1.$$ We also have the trivial identity 
\begin{equation}
\sum_{k\in \tfrac{1}{2}\mathbb{N}} g_{k} =1.
\end{equation}
In addition to these we use the results form \cite{FGL} that:
\begin{itemize}
\item $g_{0.5}= 0.5- \tfrac{2}{\pi^2}$
\item $\tfrac{4}{\pi^2} \leq g_1 \leq 0.5.$
\end{itemize}
Let us assume that $g_1= 0.5 - y$ for $ y \geq 0$ and $g_{1.5}=0.1-x$ for $x \in \mathbb{R}.$ We will show that $x$ must be smaller than $-0.0079271.$  We have that $$ g_2+ g_{2.5}+ \cdots= \tfrac{2}{\pi^2}-0.1+x+ y.$$ By using \eqref{6.1} we have
\begin{equation}
\label{6.3}
1.5(0.1-x)+ 2(\tfrac{2}{\pi^2}-0.1+x+ y) \leq 0.25 + \tfrac{1}{\pi^2} +y,
\end{equation}
and by simplifying the above we get
\begin{align}
\label{6.4}
  0.5x+ y \leq 0.3-\ \tfrac{3}{\pi^2} \leq -0.0039635,
\end{align}
which gives $x \leq -0.0079271.$ This will proves the first part of the proposition. To prove the second part assume $g_{1.5}=0.1+x$ for some positive $x$. By using a similar calculation as in \eqref{6.4} we have that
\begin{align}
\label{6.5}
&  y-0.5x \leq -0.0039635.
   \\ & g_{1.5}= 0.1 + x \geq 0.1+ 2(y+ 0.0039635), \notag
\end{align}
which by noting that $g_1= 0.5-y$ gives the second part. \\

For the third part, by \eqref{6.1} we have that $$1.5(0.1+x)+ 2g_2 +2.5g_{2.5} + \cdots = 0.25 + \tfrac{1}{\pi^2} +y,$$ which implies
\begin{align}
\label{6.6}
& 2g_2 +2.5g_{2.5} + \cdots = 0.1+\tfrac{1}{\pi^2}+y-1.5 x.
\end{align}

 Using \eqref{6.5} we get $y-1.5x< -2y-0.0118905, $ and inserting this inequality in the above equation finishes the proof of the third part of the proposition. For the last part assuming $g_1= \tfrac{4}{\pi^2},$ then we have that $y=0.094715\cdots$ which by using \eqref{6.5} implies that $x\geq 0.19735.$ Putting this into \eqref{6.6} we have the last part.\\
\end{proof}

Now we will give the proof of Theorem \ref{Main-Th}. We begin by considering the series expansion of $C_{1,\alpha}(t)$ that we defined in \eqref{C_{1, alpha}}. We have that (see in \cite[proof of Lemma 1]{Gonek})
\begin{align}
\label{C-expansion}
\notag C_{1,\alpha}(t) &= -\frac{1}{\alpha \log T}  \sum_{n<T^{\alpha}}\bigg( \frac{\Lambda(n)}{n^{\tfrac{1}{2}+it}} + \frac{\Lambda(n)}{n^{\tfrac{1}{2}-it}}\bigg) \Big(1-\frac{\log n}{\alpha \log T}\Big) \\& + 2\Re \frac{T^{\alpha(\tfrac{1}{2}-it)}}{(\tfrac{1}{2}-it)^2 \alpha^2\log^2 T} + O\big(\frac{1}{\log T}\big).
\end{align}
Similarly we have 
\begin{align*}
\notag C_{2,\alpha}(t) = -\frac{1}{\alpha \log T}&  \sum_{n<T^{\alpha}}\bigg( \frac{\Lambda(n)}{n^{\tfrac{1}{2}+it}} + \frac{\Lambda(n)}{n^{\tfrac{1}{2}-it}}\bigg) \Big(1-\frac{\log n}{\alpha \log T} +\frac{\sin(2\pi\tfrac{\log n}{\alpha \log T})}{2\pi} \Big) \\& + 2\Re \frac{T^{\alpha(\tfrac{1}{2}-it)}}{(\tfrac{1}{2}-it)^2 \alpha^2\log^2 T} + O\big(\frac{1}{\log T}\big).
\end{align*}

For detailed proof of \eqref{C-expansion} see \cite{Me}. We continue with giving the proof with $C_{1,\alpha}$ since $C_{2,\alpha}$  would be very similar. Let $T_0=T\log^{-2}T$ and first let us consider the effect of the poles of the Riemann zeta function in the above. We have that
\begin{align*}
\int \omega(\tfrac{1}{2}+it)\frac{T^{\alpha(\tfrac{1}{2}-it)}}{(\tfrac{1}{2}-it)^2 \alpha^2\log^2 T}\Big|\sum_{n<T_0} \frac{a(n)}{n^{\tfrac{1}{2}+it}}\Big|^2.
\end{align*}
For $\alpha<2$ the above is $O(T^{-1})$. For $\alpha>2$ it comes down to considering
\begin{align}
\label{smal-fourier}
\int \omega(\tfrac{1}{2}+it)\frac{1}{(\tfrac{1}{2}-it)^2 } \big(\frac{s}{rT^\alpha}\big)^{\tfrac{1}{2}+it}.
\end{align}
We can look at the above as the Fourier transform of $\omega(\tfrac{1}{2}+it)(\tfrac{1}{2}-it)^{-2} $ at $\log s -\log rT^\alpha.$ Since $|\log s -\log rT^\alpha| \geq \log T$ and the function is smooth we have that \eqref{smal-fourier} is very small. With the above explanation and using \eqref{C-expansion} we have that the expectation of $C_{1,\alpha}$ with respect to $\mu_A$ equals
\begin{align}
\label{proof-thm1}
 -\notag \int \frac{\omega(\tfrac{1}{2}+it)}{\alpha \log T} &\sum_{r,s<T_0} \frac{a(r)a(s)}{s}\big(\frac{s}{r}\big)^{\tfrac{1}{2}+it}  \sum_{n<T^{\alpha}}\bigg( \frac{\Lambda(n)}{n^{\tfrac{1}{2}+it}} + \frac{\Lambda(n)}{n^{\tfrac{1}{2}-it}}\bigg) \Big(1-\frac{\log n}{\alpha \log T}\Big)\\ &  \notag=-\frac{2}{\alpha \log T}  \displaystyle{ \sum_{\substack{r,s<T_0 \\ n<T^{\alpha}}} \frac{a(r)a(s)\Lambda(n)}{ns}\Big(1-\frac{\log n}{\alpha \log T}\Big) \int \omega(\tfrac{1}{2}+it)\big(\frac{sn}{r}\big)^{\tfrac{1}{2}+it}} \\ & = -\frac{2}{\alpha \log T}\displaystyle{\sum_{\substack{rn<T_0 \\ n<T^{\alpha}}} \frac{a(rn)a(r)\Lambda(n)}{rn}\Big(1-\frac{\log n}{\alpha \log T}\Big)},
\end{align}

plus a small error term. For the proof of Theorem \ref{Main-Th} we consider $$a(n)= \lambda_{\beta_1, \beta_2}(n)d_{r}(n)\big(1-\frac{\log(n)}{\log T}\big)^{\eta},$$ as in \eqref{def_an}. Substituting this in \eqref{proof-thm1} and considering the shift by $d$ we need to estimate 
\begin{equation*}
r\sum_{p<T} \frac{\lambda_{\beta_1, \beta_1}(p) \log p }{ p^{1-id}} \big(1-\frac{\log p}{\alpha\log T}\big) \sum_{m < T_0/p} \frac{d^2_{r}(m)}{m}\big(1-\frac{\log m}{\log T}\big)^\nu \big(1-\frac{\log m}{\log T}-\frac{\log p}{\log T}\big)^\nu.
\end{equation*}
We use the following on the sum of the generalized divisor function~\cite{BMN}
\begin{equation}
\sum_{m< x} \frac{d^2_{r}(m)}{m} = A_r(\log x)^{r^2} + O\big( (\log T)^{r^2-1}\big),
\end{equation}
and the prime number theorem to get  
\begin{equation*}
r \int_{1}^{T} \frac{\lambda_{\beta_1, \beta_1}(x) }{ x^{1-id}} \big(1-\frac{\log x}{\alpha \log T}\big) \int_{1}^{ T/x} A_r \frac{r^2 (\log y)^{r^2-1}}{y}\big(1-\frac{\log y}{\log T}\big)^\eta \big(1-\frac{\log y}{\log T}-\frac{\log x}{\log T}\big)^\eta dx dy.
\end{equation*}
By changing variable $u= \log x/ \log T$ and $v= \log y/ \log T,$ we get
\begin{equation*}
r^3 A_r (\log T)^{r^2+1} \int_{0}^{1}\int_{0}^{1-u} \lambda_{\beta_1, \beta_1}(u) e^{2 \pi i d u} v^{r^2-1} (1-u/\alpha)(1-v)^\eta(1-u-v)^\eta.
\end{equation*}
Following a similar argument we have 

\begin{equation*}
\sum_{m<T} \frac{|a(m)|^2}{m} = r^2 A_r (\log T)^{r^2} \int_{0}^{1}  v^{r^2-1} (1-v)^{2\eta} dv.
\end{equation*}

\begin{proof}[Proof of Corollary \ref{Corr1}] We apply Theorem \ref{Main-Th} with $i=2, \hspace{1 mm}  r=1.8, \hspace{1 mm}\alpha=1, \hspace{1 mm}\nu=0.4$ and $d=0$ and we get $$
\int C_{2,1}(t)d\mu_A(t)= 0.8226\cdots .$$
Therefore there must exist $t$ such that $C_{2,1}> 0.8226.$ If we have normalize gaps smaller than half, then we are done. If not, based the definition of $C_{2,1}$ we have 
\begin{equation*}
    C_{2, 1}(t) < \hat{C}_2(t) + \hat{C}_2(x-t)+ \hat{C}_2(x+y-t) + \hat{C}_2(z+t) + \hat{C}_2(z+w+t),
\end{equation*}
where $x$ is the size of the gap that contains $t$, $y$ is the size of the next gap, $z, w$ are sizes of the previous gap and the one before. Recall that $$\hat{C}_{2}(x)= \big(\frac{\sin(\pi x)}{\pi x}\big)^2\big(\frac{1}{1-x^2}\big).$$ We also used that fact that for $x>1$, we have that $\hat{C}_{2}(x)<0$ to establish the inequality. Now our problem turns to the following  optimization problem:
 $$\text{Maximize} \hspace{2 mm}\hat{C}_2(t) + \hat{C}_2(x-t)+ \hat{C}_2(x+y-t) + \hat{C}_2(z+t) + \hat{C}_2(z+w+t),$$
 
 with conditions $$x,y,z,w>0.5,$$ and $$x+y, x+w, w+z \geq 1.181,$$ and $t<x.$ We solve this using Python and we got that the maximum is smaller than $0.82207\cdots$, which is obtained by $x,y=0.59049\cdots$ and $z, w= 1.409160\cdots.$ This shows that we must get two consecutive gaps of length $x, y$ with $x+y< 1.181$ to ensure $C_{2,1}> 0.8226.$ This completes the proof.
\end{proof}
\section{\bf A new approach to gaps between zeta zeros}
\label{refine}
Regarding AH and the Landau-Siegel zeros problem, there is a common speculation that suggest in order to resolve the problem we should know about how to handle off-diagonal terms that arise form long mollifiers. 
In this section we prove results that, to some extent, confirms this speculation. We also suggest a remedy that may help to get better bounds from our method and possibly resolve the problem. \\

In section \ref{ValueDis} we estimated mean-values of our test functions with respect to mollifier-measures built using Dirichlet polynomials. We mentioned that if we can show that
\begin{equation}
\label{star1}
\int C_{1,1} d\mu_A >1,    
\end{equation}
then this would reject the alternative hypothesis. In this section we show that this is not possible using Dirichlet polynomials of length smaller than $T^{1-\epsilon}.$ For  Dirichlet polynomials of length $> T^{1-\epsilon}, $ estimating \eqref{star1} requires off-diagonal estimation, which usually is a difficult problem on its own. \\

An important point is, in order to reject AH, it is not necessary to prove \eqref{star1}. We only need to show that there exist $t \sim T,$ that $C_{1,1}(t)>1.$ This is the place where the distribution of $\mu_A$ become relevant. For example, to understand, the distribution of values of $C_{1,1}$ around its mean,  we can ask, what portion of $\mu_A$ is distributed on $$\{t: C_{1,1}(t)> \sigma \}.$$
In Theorem \ref{Prop6} we partially answered this question for the Lebesgue measure. \\

In this section we also provide a partial answer to this question for $\mu_A,$ with a general mollifier. Further, we explain that how would understanding the distribution of $\mu_A$ can help us with the problem of gaps between zeros. \\

Let us begin by explaining the method of Montgomery and Odlyzko~\cite{Mo} for detecting small gaps between zeros of the zeta function. Recall that $\mu_A$ is defined in \eqref{def-Mu}. Using analytic methods one can show that
\begin{equation}
\label{Mu-thm-MO}
\sum_{\zeta(1/2+i\gamma)=0}\mu_A\Big( \big(\tilde{\gamma}-\frac{c}{2}, \tilde{\gamma}+\frac{c}{2}\big] \Big) \cong c-\frac{2\sum_{mp<T} \frac{a(m)a(mp)}{mp} \sin \big(\pi c\frac{ \log p}{\log T}\big)}{\pi \sum_{m<T}\frac{|a(m)|^2}{m}}.
\end{equation}

To proceed, let us assume that we do not have any (normalized) gap smaller than $c$. Based on this assumption intervals in \eqref{Mu-thm-MO} are disjoint, therefore the sum equals the measure of the union of these intervals. By the definition, $\mu$ is a probability measure and the LHS of \eqref{Mu-thm-MO} must be smaller than $1.$ Choosing $a(n)=\lambda(n)$ would flip the negative sign in \eqref{Mu-thm-MO} and by the prime number theorem we have the RHS is about $2c- 0.0276\pi^2 c^3$ which by setting $c=0.5192,$ is greater than $1.$ \\

Improvements in~\cite{CGG, BMN, FWu, Pre} that we previously mentioned are results of various efforts to maximize the quantity
\begin{equation}
\label{eq5}
  \mu_{\text{positive}, A} := \frac{\sum_{mp<T} \frac{a(m)a(mp)}{mp} \sin \big(\frac{\pi \log p}{2\log T}\big)}{\sum_{m<T}\frac{|a(m)|^2}{m}},
\end{equation}
over different Dirichlet polynomials $A.$ 
\\

The disadvantage of this method is that no matter the choice of $a(\cdot),$ we have that $\mu_{\text{positive}, A}$ is always smaller than $0.5,$ which means that the method is unable to prove existence of normalized gaps smaller than $0.5.$ \\
 
Our objective in this part to introduce a refinement of the previous method which can be used on the problem of large values of the test function we considered. This is an step forward toward understanding of how we may resolve AH. \\
 
Our first observation is that one way to reject  AH, is to maximize
\begin{equation}
\label{eq2}
   \mathds{E}_{C,A}:=  \frac{2\sum_{mp<T} \frac{a(m)a(mp) \log p}{mp} \hspace{1 mm}C\big( \frac{\log p}{\log T}\big)}{\log T\sum_{m<T}\frac{|a(m)|^2}{m}},
\end{equation}
while minimizing $\mu_{\text{positive}, A}$ in \eqref{eq5}. The logic behind it is that \eqref{eq2} is the expectation of a certain test function (denoted below by $C,$ which we defined in \eqref{C_{1, alpha}} and \eqref{C_{2, alpha}}) with respect to the probability measure induced by $A(\tfrac{1}{2} + it)$ as the mollifier, i.e.
\begin{equation}
\int C(t) d\mu_{A}.
\end{equation}
By assuming AH, we can put some restriction on the distribution of $\mu_{\text{positive}}.$ Therefore, on one hand it can help us to get an upper bound for $C(\cdot)$ and on the other hand, by the Theorem \ref{Main-Th} we have the exact expectation of $C(\cdot)$ with respect to $\mu_A$ . Therefore by minimizing $\mu_{\text{positive}, A}$ and maximizing $ \mathds{E}_{C,A}$ we can amplify the effect of AH in order to reach a contradiction.  \\
\subsection{How to contradict the alternative hypothesis}From the last two paragraphs and assuming AH, we see that the $\mu_A$-measure of the region where $C_{1,1}(t)$ is positive is about $0.5+\mu_{\text{positive}, A}.$ Consider this with with the fact that simplicity hypothesis and AH implies that $C_{1, 1}< 1,$  we get $\mathds{E}_{C_{1, 1}, A}< 0.5 + \mu_{\text{positive}, A}. $  Thus, if we could prove that
 $$ \mathds{E}_{C_{1, 1}, A} \geq 0.5 +\mu_{\text{positive}, A}  $$ that would reject  AH.
 Obviously the more we know about $\mu_A$ under AH, the easier it gets to bound the test function and compare it with the expectation we get from the theorem. 

\subsection{Advantages over the previous method} Here we list some advantages of the refinements compared to the previous method. \\

\begin{itemize}
\item There is no obvious reason the refined method cannot contradict the AH, unlike the Montgomery-Odlyzko method which is inherently unable to resolve the problem. 
\item The Montgomery-Odlyzko method, is designed to get results on small/large gaps and not specifically to resolve AH. Disproving AH is easier than proving there are infinitely many normalized gaps smaller than $0.5.$ 
\item Using the Liouville function is not necessary. In the Montgomery-Odlyzko method to get to small gaps, it is essential to define $a(n)$ as $\lambda(n)$ times a positive function. In our method having or not having $\lambda$ involved, does not make a big difference. This become important if one would like to look at longer Dirichlet polynomials. Considering them requires off-diagonal estimation. If the Liouville function was used with a larger Dirichlet polynomials then this would require assuming, at a minimum, the strong version of Chowla's conjecture. 
\end{itemize}
\section{\bf Statement of Results} As one of our  goals, we are seeking to maximize,$ \mathds{E}_{C_{1, 1}, A}$, in \eqref{eq2} while minimizing the measure of the region in which the test function is large. In this direction the first question we answer is, what would the best bound (for Dirichlet polynomials of length $\leq T^{1-\epsilon}$) we can hope for? The ``trivial'' bound is $$\mathds{E}_{C_{1, 1}, A}= \int {C_{1, 1}}(t) d\mu_{A} \leq 1.$$ 
We prove a stronger result:
\begin{theorem}
\label{Upp}
Let $A(s)$ be a Dirichlet polynomials of length smaller than $T^{1-\epsilon}$ and $\mu_A$ defined as \eqref{def_A}. We have that 
\begin{equation}
\mathds{E}_{C_{1, 1}, A}=\int C_{1, 1}(t) d\mu_{A} 
 < 0.79371.
\end{equation}
Also, we have 
\begin{equation}
\mathds{E}_{C_{2, 1}, A}= \int C_{2, 1}(t) d\mu_{A} 
 < 0.90156.
\end{equation}

\end{theorem}
The theorem basically, says that $C_{1, 1}$ is expected to be smaller than $0.79371,$ for any probability measure that we make using Dirichlet polynomials  of length smaller than $T^{1-\epsilon}$.\\

In terms of lower bound the best we have so far is 
$$ \mathds{E}_{C_{1, 1}, A} \geq 0.74097\cdots,$$

where $A$ is a Dirichlet polynomials with coefficients  $$a(n)= \lambda(n)d_{1.8}(n)\big(1-\frac{\log(n)}{\log T}\big)^{0.4}.$$  For this choice $\mu_{\text{ positive}, A}$ is about $0.94787.$ \\

To make a comparison, if we choose $a(n)=\lambda(n),$ it would give that \eqref{eq2} is $0.6666\cdots$ and \eqref{Mu-thm-MO} is about $0.967\cdots.$ \\

\subsection{The Maximize vs Minimize Problem}


 As we discussed, assuming AH we can get an upper bound for the expectation of $C_{i,\alpha}$ with respect to $\mu_A.$ Therefore our strategy to contradict it is to make a measure $\mu_A$ such that  the expectation of $C d\mu_A$ under AH is smaller than its expectation from Theorem \ref{Main-Th}. \\
 
 So far  we only tried the conventional choices for $a(\cdot).$  The best choice I found is  
$$a(n)= d_{1.4}(n)\big(1-\frac{\log(n)}{\log T}\big)^{0.2},$$
for which the expectation of the test function is about $0.73.$ Assuming AH, about $96\%$ of the measure is where the test function is positive and at most $58\%$ of the measure can be where the test function attains values larger than $0.7$, which are on consecutive gaps of length $0.5$. With a minor calculation we have that in order to have it consistent with  AH we must have $C(t)> 0.9$ for almost all of the $58\%$ of the measure on which the test function can be large. For $C(t)$ to be bigger than $0.9$, $t$ should vary inside at least four consecutive $0.5$-gaps. \\

It seems to the author that the best results may come from non-conventional choices of $a(\cdot)$. For example consider a multiplicative function of this sort

\begin{equation}
f(p):= \begin{cases}
1.4, \text{ if } p< T^{0.4} \text{ and } T^{0.6}<p< T\\
10,  \hspace{1 mm}\text{ if }  T^{0.6} \leq p\leq T^{0.6}.
\end{cases}
\end{equation}
Having large values for primes around $\sqrt{T}$ would decrease the measure of the region on which $C_{i, 1}(t)$ is large. We conclude this section with some remarks. 
\begin{Rem}
We assumed the Riemann hypothesis in our results. 
It is important to note that, for a prospective application to the Landau-Siegel zero problem, because of the Deuring-Heilbronn phenomenon the assumption of the Riemann hypothesis in the theorem may be relaxed a bit.
\end{Rem}
\begin{Rem}
For $\alpha \geq 1,$ AH forces $C_{1,\alpha}(t) $  to be smaller than $1$, while we know that $C_{1,\alpha}(t) $ is bigger than
one near to consecutive gaps of small size. Therefore, assuming there are gaps of small size, we expect that 
that the moments \begin{equation}
\int C_{1,\alpha}(t)^{k} d\mu_{A}
\end{equation}
are bounded away from zero, whereas under AH they tend to zero as $k
\to \infty$.  Estimating these
moments seems difficult, as there will be complicated off-diagonal
terms. We will continue our investigation in~\cite{2nd-Pap} with particular focus on  higher moments.

\end{Rem}
\section{\bf Discussion and Questions} There are similarities between maximizing $   \mathds{E}_{C,A}$ in \eqref{eq2} and maximizing 
\begin{equation}
\label{Res}
    \displaystyle{ \frac{\sum_{mk<T} \frac{a(m)a(mk)}{mk} }{\sum_{m<T}\frac{|a(m)|^2}{m}}},
\end{equation}
which is the essence of the resonance method for finding large values of the Riemann zeta function. Soundararajan in \cite{Sound} found the optimal choice for Dirichlet polynomials of length $<T^{1-\epsilon}.$ Interesting work was done by Bondarenko and Seip~\cite{BSe}. They used Dirichlet polynomials of length bigger than $T$ and dealt with the off-diagonal terms in a following way: by using an appropriate weight they made sure that contributions of off-diagonal terms are positive. Then they could throw them out and have a lower bound which improved Soundararajan's result. Theoretically, one can use a similar method here, however we have to to avoid using the Liouville function and therefore our measure would likely be concentrated on large gaps. This brings us to the following question:
\begin{Question} 
Is it possible to a construct a Dirichlet polynomial, A,  of arbitrary length such that $$\int C_{1, 1}(t) d\mu_A \rightarrow -1,$$ as the length goes to infinity?
\end{Question}
Proving such a result may not immediately contradict AH, however it would settle the large gap conjecture i.e. gaps between zeros of the zeta function get arbitrary large. \\


From Theorem \ref{Upp} we can conclude that if the length of $A$ is smaller than $T^{1-\epsilon}$, then $\mu_A$ is unable to detect the microscopic behaviour of zeros. For example the theorem shows that we cannot find $A$ such that $\mu_A$ has more then $50\%$ of its support concentrated on two consecutive gaps of length $0.25.$  \\

Now let $A$ be a Dirichlet polynomial of arbitrary length such that in calculation of $$\int C_{1,1}(t) d\mu_A $$ the contribution of off-diagonal terms are negligible. We call $A$ a sparse Dirichlet polynomial. We conclude this section with a question that weather or not we can build a measure using sparse Dirichlet polynomials that beats the upper bound in Theorem \ref{Upp}.
\begin{Question} 
Is it possible to find a sparse Dirichlet polynomial, A, of arbitrary length such that $$\int C_{1, 1}(t) d\mu_A >  0.79371.$$ 
\end{Question}

\section{\bf Interpretation of $\mu_A$ Under the Alternative Hypothesis.} 
\label{IntmuAh}
AH would approximately determine the distribution of the measures $\mu_A$ defined in \eqref{def-Mu}. In this section we will use the following lemmas as a tool to study our measures. 
\begin{lemma}
\label{Measure}
Let $A(s)$ be as \eqref{def_A}. Assuming the Riemann hypothesis 
\begin{align}
\label{Measure-interval}
\notag
\sum_{\gamma} \mu_A\big[\hspace{0.5 mm}(\hspace{0.5 mm}\gamma+ &\alpha\frac{2\pi}{\log T}, \gamma+  \beta\frac{2\pi}{\log T}\hspace{0.5 mm}]\hspace{0.5 mm}\big]= \beta -\alpha \\ & - \frac{1}{\pi} \sum \frac{a(mp)\overline{a(m)}}{mp}\bigg(\sin(2\pi\beta\frac{\log p}{\log T})- \sin(2\pi\alpha\frac{\log p}{\log T}) \bigg) +O(\frac{1}{\log T}).\\ \notag
\end{align}
\end{lemma}
\noindent The next lemma is about our specific choice of $a(n)$ in \eqref{def_an}.
\begin{lemma}
\label{lemm-meas}
Let $\mu_{A_{\beta_1, \beta_2, r , \eta}}$ be the measure we build, as in \eqref{def_A}, using $$a(n)= \lambda_{\beta_1, \beta_2}(n)d_{r}(n)\big(1-\frac{\log(n)}{\log T}\big)^{\eta}.$$ Assuming the Riemann hypothesis  
\begin{align}
\notag
\sum_{\gamma} \notag \mu_{A_{\beta_1, \beta_2, r , \eta}}\big[\hspace{0.5 mm} (\hspace{0.5 mm}& \gamma+ \alpha\frac{2\pi}{\log T}, \gamma+  \beta\frac{2\pi}{\log T}\hspace{0.5 mm}]\hspace{0.5 mm}\big]= \beta - \alpha + +O(\tfrac{1}{\log T}) \\ & \notag - \frac{r \int_{0}^{1}\int_{0}^{1-u} \frac{\lambda_{\beta_1, \beta_2}(u)}{u} \big(\sin(2\pi\beta u)- \sin(2\pi\alpha u) \big)  v^{r^2-1}(1-v)^\eta(1-u-v)^\eta }{\pi \int_{0}^{1}v^{r^2-1}(1-v)^{2\eta}}.
\end{align}
\end{lemma}

 To proceed, we will use the following notations:
\begin{itemize}

\item By $\sigma$-gap we mean a normalized gap of length equal to $\sigma$.
\item By $  \mu_{\alpha_2,\alpha_1 ,\sigma , \beta_1,\hspace{1 mm} x \in I } $ we mean the measure of sub-interval $I$ inside a $\sigma$-gap that is followed by, respectively, an $\alpha_1$ and an $\alpha_2$-gaps. The next gap after $\sigma$ equals $\beta_1.$ If we put a dot in any of these places  we mean that we have no restriction on that gap or interval. \\

\end{itemize}

\subsection{Distribution of $\mu_A$ Under AH }
We continue with providing a manual as how we can get information about the distribution of $\mu_A$ under the assumption of AH.
\begin{enumerate}
\item \textbf{Measure of region around each zeros.} Using Lemma \ref{Measure} we can find a precise estimate for 
\begin{equation}
\mu_{A, \pm c} :=\mu_A\big(\cup \hspace{1 mm} [\tilde{\gamma}-c, \tilde{\gamma}+c]\big)
\end{equation}
We define $\mu_{A, \text{Out}}:= 1- \mu_{A, \pm 0.25}.$
\item \textbf{Measure of $0.5$-gaps.} To estimate the measure of $0.5$-gaps we first need to calculate 
\begin{equation}
\label{1st-shft}
\mu_A\big(\cup \hspace{1 mm} \tilde{\gamma} \pm [0.25,  0.5]\big).
\end{equation} This will cover all $0.5$-gaps plus some middle parts of gaps $\geq 1,$ which is smaller than $\mu_{A, \text{Out}}$. Therefore we get that
$$\mu_A\big(\cup \hspace{1 mm} \tilde{\gamma} \pm [0.25,  0.5]\big) - \mu_{A, \text{Out}} \leq \mu_A(0.5-\text{gaps}) \leq  \mu_A\big(\cup \hspace{1 mm} \tilde{\gamma} \pm [0.25,  0.5]\big).$$ This would yield a good estimates if $\mu_{A, \text{Out}}$ is small.
\item \textbf{Measure of region that the test function attain high values.} 
The region where $C_{1, 1}(t)$ is large  often come after (or before) a $0.5$-gap. For example two or three consecutive gaps or length $0.5.$ Consider the following events $E_1: \cup \hspace{1 mm}\tilde{\gamma}+ (0, 0.5)$ and $E_2: \cup \hspace{1 mm}\tilde{\gamma}+ (0.5, 1).$ We have that $E_1 \cup E_2$ contains all the intervals except parts of gaps $\geq 1.5$. On the other hand we have $E_1 \cap E_2$ contains intervals of length $0.5$ that comes after a $0.5$-gap. By using $\mu(E \cup F)=\mu(E)+\mu(F)- \mu(E \cap F)$ we can get an estimate of measure $0.5$-intervals that comes after a $0.5$-gap.  
\item \textbf{Measure of isolated $0.5$-gaps.} Let $E_1: \cup \hspace{1 mm}\tilde{\gamma}+ (0.5, 1)$ and \break $E_2: \cup \hspace{1 mm}\tilde{\gamma}- (0.5, 1).$ We have that $E_1 \cup E_2$ contains all of the \break $0.5$-gaps except the ones that comes before and after $\geq 1$-gaps. On the other hand we have $E_1 \cap E_2$ contains $0.5$-gaps that comes after and before $0.5$-gaps. Again by using $\mu(E \cup F)=\mu(E)+\mu(F)- \mu(E \cap F)$ we can get an estimate on the measure of isolated $0.5$-gaps. More precisely using the notation above, we get 
\begin{align*}
\mu_{\cdot,0.5 ,0.5 , 0.5, \cdot }& + \mu_{\cdot,0.5 ,1 , \cdot, x<0.5 }+ \mu_{\cdot,\cdot ,1 ,0.5, x>0.5}+ \mu_{\cdot,\cdot ,1.5 , \cdot, 0.5<x<1 }=\\& \notag |E_1|+|E_2|-1  +2\mu_{\cdot,\geq 1 ,\geq 1.5 , \cdot, x<0.5 } + \mu_{\cdot,\geq 1 ,0.5 , \geq 1, \cdot }
\end{align*}
\end{enumerate}

\begin{Rem}
\label{Rem-Liou}
We will quickly summarize what we get if we apply the above to the basic choice $a(n)=\lambda(n)$. Roughly speaking, we get that the measure of the $0.5$-gaps is about $0.77$ and the measure of $\geq 1$-gaps about $0.23.$
Around $40\%$ of the measure is distributed in $0.5$-gaps that are followed by another $0.5$-gap. There is a direct correlation between the measure of the $0.5$-gaps with next and previous gaps equal to $0.5,$ and  the $0.5$-gaps such that next and previous gaps are $\geq 1.$ 
 \end{Rem}
 \section{Proof of Lemma \ref{Measure} and of Theorem \ref{Upp}}
 We begin with giving a proof of the fact that under AH, $C_{1, 1}$ is always smaller than $1$: \\
 
 \begin{lemma} Assume the alternative hypothesis, we have 
 \begin{equation}
     C_{1, 1}(t) \leq 1. 
 \end{equation}
Moreover, 
  \begin{equation}
  \label{8.2}
     C_{1, 2}(t) \leq 0.5. 
 \end{equation}
 \end{lemma}
 
 \begin{proof}
 Recall that, 
 \begin{equation*}
C_{1, \alpha}(t)= -\alpha^{-1} +\sum_{\gamma} \bigg(\frac{\sin(\tfrac{\alpha}{2}(\gamma-t)\log T)}{\tfrac{\alpha}{2}(\gamma-t)\log T}\bigg)^2.
\end{equation*}
If we assume AH then the maximum of $C_{1, \alpha}$ happens if we have an infinite sequence of $0.5$-gaps. Therefore
 \begin{equation*}
C_{1, 1}(t) \leq -1 + \sum_{k \in \mathbb{Z}} \bigg(\frac{\sin(k\tfrac{\pi}{2}+ d)}{k\tfrac{\pi}{2}+ d}\bigg)^2.
\end{equation*}
 The Fourier transform of $\big(\sin(\pi v/2)/(\pi v/2) \big)^2$ is $2(1- 2|\xi|)$ for $|\xi|< 1/2$ and it is $0$ elsewhere. Therefore using the Poisson summation formula, 
 \begin{equation*}
     \sum_{k \in \mathbb{Z}} \bigg(\frac{\sin(k\tfrac{\pi}{2}+ d)}{k\tfrac{\pi}{2}+ d}\bigg)^2= \sum_{m \in \mathbb{Z}} 2(1-2|m|) \chi_{|m|<1/2}e^{2\pi i d m}=2.
 \end{equation*}
 Therefore, $C_{1, 1}(t) \leq 1.$ The proof of \eqref{8.2} follows similarly.
 \end{proof}

Now we give proofs of Lemma \ref{Measure} and of Theorem \ref{Upp}. After applying Lemma \ref{Measure} , the proof of Lemma \ref{lemm-meas} goes similarly to proof of Theorem \ref{Main-Th}.
\begin{proof}[Proof of Lemma \ref{Measure}]
We need to calculate the following
\begin{equation}
\label{eq-meas}
\displaystyle{\frac{\sum_{\gamma} \omega(\tfrac{1}{2}+i\gamma) \bigintss_{2\alpha \pi \log^{-1}T}^{2\beta \pi \log^{-1}T} \Big|A(\tfrac{1}{2}+i\gamma+ix)\Big|^2 dx}{\int \omega(\tfrac{1}{2}+ix) \big|A(\tfrac{1}{2}+ix)\big|^2 dx}},
\end{equation}
\\
for $A(s)=\sum a(n)n^{-s}\mathds{1}_{[n \leq T\log^{-2} T]}$, we expand the square in the numerator and \eqref{eq-meas} comes down to calculating
\begin{equation*}
\sum_{\gamma} \omega(\tfrac{1}{2}+i\gamma) \sum_{m, n} \frac{a(m)\overline{a(n)}}{n} \Big(\frac{n}{m}\Big)^{\tfrac{1}{2}+i\gamma} \bigintsss_{2\alpha \pi \log^{-1}T}^{2\beta \pi \log^{-1}T} \Big(\frac{n}{m}\Big)^{ix}.
\end{equation*}
For the $m=n$ we get 
$$\sum_{\gamma} \omega(\tfrac{1}{2}+i\gamma) \sum_{m} \frac{|a(m)|^2}{m} \big(\frac{2\pi\beta}{\log T}-\frac{2\pi\alpha}{\log T}\big).$$ Now it is easy to verify that $$\sum_{\gamma} \omega(\tfrac{1}{2}+i\gamma)= \frac{\log T}{2\pi}+O(\frac{1}{T}),$$ therefore we obtain
$(\beta-\alpha)\sum |a(m)|^2 m^{-1}$, for the diagonal terms. For $m \neq n $ we apply the smooth version of the Landau-Gonek formula, \cite{Me}[Lemma 2.1], and we get
$$\frac{1}{\pi} \sum_{m, p} \frac{a(mp)\overline{a(m)}}{mp}\bigg(\sin\Big(2\pi\beta \frac{\log p}{\log T}\Big)- \sin\Big(2\pi\alpha \frac{\log p}{\log T}\Big)\bigg),$$
plus a negligible error term. This completes the proof of the lemma.
\end{proof}
\begin{proof}[Proof of Theorem \ref{Upp}] We begin by writing the numerator of $ \mathds{E}_{C_{1, 1}, A}$, in \eqref{eq2} as 
\begin{equation}
\label{Cauchy1}
\frac{2}{\log T} \sum_{mp<T} \bigg(\frac{a(m)\sqrt{\log p}}{\sqrt{mp}}\big(1- \frac{\log p}{\log T}\big)\bigg)\bigg(\frac{a(mp)\sqrt{\log p}}{\sqrt{mp}}\bigg).
\end{equation}
By Cauchy–Schwarz inequality we have \eqref{Cauchy1} is smaller than 
\begin{align}
\label{Cauchy2}
\notag \frac{2}{\log T}  \bigg(\sum_{mp<T} & \frac{|a(m)|^2 \log p}{mp}\big(1- \frac{\log p}{\log T}\big)^2 \bigg)^{\tfrac{1}{2}}  \bigg(\sum_{mp<T} \frac{|a(mp)|^2 \log p}{mp}\bigg)^{\tfrac{1}{2}}\\ \notag
\leq \frac{2}{\log T} & \bigg(\sum_{m<T} \frac{|a(m)|^2 }{m} \sum_{p<T/m}\frac{\log p}{p}\big(1- \frac{\log p}{\log T}\big)^2 \bigg)^{\tfrac{1}{2}}  \bigg(\sum_{m<T} \frac{|a(m)|^2 }{m} \log m\bigg)^{\tfrac{1}{2}}\\
\sim \frac{2}{\sqrt{3}} & \bigg(\sum_{m<T} \frac{|a(m)|^2 }{m} \Big(1-\big( \frac{\log m}{\log T}\big)^3\Big) \bigg)^{\tfrac{1}{2}}  \bigg(\sum_{m<T} \frac{|a(m)|^2 }{m} \big(\frac{\log m}{\log T}\big) \bigg)^{\tfrac{1}{2}}.
\end{align}
To continue for $0 \leq i \leq k-1,$ we set 
\begin{equation}
\frac{\displaystyle{ \sum_{T^{i/k}< m< T^{(i+1)/k}} {|a(m)|^2m^{-1}}}}{\displaystyle{\sum_{m<T}|a(m)|^2m^{-1}}}= A_i
\end{equation}
Using this partition, we have that \eqref{Cauchy2} is smaller than 
\begin{equation*}
\frac{2}{\sqrt{3}} \sum_{i=1}^{k} \bigg(A_i\Big(1- \big(\frac{i}{k}\big)^3 \Big) \bigg)^{\tfrac{1}{2}} \bigg(A_i\frac{i+1}{k} \bigg)^{\tfrac{1}{2}} = \frac{2}{\sqrt{3}} \sum_{i=1}^{k} A_i \sqrt{\Big(1- \big(\frac{i}{k}\big)^3 \Big)\frac{i+1}{k}}.
\end{equation*}
To finish the proof we use the fact that $\sum A_i =1$ and note that $$\max \frac{2}{\sqrt{3}} \sqrt{(1-x^3)x}= 0.79370\cdots.$$ 
For $$\int C_{2, 1} d\mu_A,$$
the method is similar. First change appears in \eqref{Cauchy2} with the following sum 
\begin{equation}
\sum_{p<T/m}\frac{\log p}{p}\bigg(1- \frac{\log p}{\log T} + \frac{\sin\big(2\pi\tfrac{\log p}{\log T}\big)}{2\pi}\bigg)^2.
\end{equation}
Therefore if we define \begin{equation}
F(x):= \int_{0}^{1-x} \Big(1-u  + \frac{\sin(2\pi u)}{2\pi}\Big)^2du,
\end{equation}
our problem comes down to find a maximum of the following $$2\sqrt{xF(x)}.$$
We solved this using Python and the max is $0.90156\cdots$ obtained for $x=0.619352.$
\end{proof}
\subsection*{Acknowledgements}
I would like to thank Simon Myerson for careful reading of the paper and his suggestions.

\end{document}